\documentclass[11pt,twoside]{article}

\usepackage{amsmath}
\usepackage{amsthm}
\usepackage{amssymb}
\usepackage{amscd}
\usepackage{setspace}
\usepackage{mathrsfs}
\usepackage{color}
\usepackage{multirow}
\usepackage{stmaryrd}	
\usepackage{mathdots}	
\usepackage{array}		
\usepackage{fancyhdr}
\input xy
\xyoption{all}
\usepackage{rotating}
\usepackage[unicode]{hyperref}	
\hypersetup{urlcolor=blue}		

\newenvironment{acknowledgements}{\list{}{\rightmargin0.75in\leftmargin0.75in}\item[]\textsc{Acknowledgements.}}{\endlist}

\newtheorem{theorem}{Theorem}[section]
\newtheorem{proposition}[theorem]{Proposition}
\newtheorem{lemma}[theorem]{Lemma}
\newtheorem{corollary}[theorem]{Corollary}

\newtheorem{theoremA}{Theorem}

\theoremstyle{definition}
\newtheorem{definition}[theorem]{Definition}
\newtheorem{remark}[theorem]{Remark}

\newcolumntype{x}[1]{>{\centering\arraybackslash$}m{#1}<{$}}

\newcommand{\mtrx}[4]{\left(\begin{array}{cc} #1 & #2 \\ #3 & #4 \\ \end{array}\right)}

\newcommand{\mf}[1]{\mathfrak{#1}}

\newcommand{\func}[4]{\begin{array}{rcl} #1 & \rightarrow & #2 \\ #3 & \mapsto & #4 \end{array} }
\newcommand{\QQ}{\mathbf{Q}}
\newcommand{\ZZ}{\mathbf{Z}}

\newcommand{\CC}{\mathbf{C}}
\newcommand{\GG}{\mathbf{G}}

\newcommand{\Zp}{\mathbf{Z}_p}
\newcommand{\Qp}{\mathbf{Q}_p}

\newcommand{\wt}[1]{\widetilde{#1}}

\newcommand{\ol}[1]{\overline{#1}}

\renewcommand{\mod}{\operatorname{mod}}

\renewcommand{\AA}{\mathbf{A}}

\DeclareMathOperator{\gl}{GL}

\DeclareMathOperator{\End}{End}

\DeclareMathOperator{\Hom}{Hom}

\DeclareMathOperator{\rec}{rec}

\DeclareMathOperator{\gr}{gr}

\DeclareMathOperator{\Gal}{Gal}
\DeclareMathOperator{\Aut}{Aut}

\DeclareMathOperator{\im}{im}

\DeclareMathOperator{\Frob}{Frob}

\DeclareMathOperator{\nr}{nr}
\DeclareMathOperator{\ord}{ord}
\DeclareMathOperator{\Sym}{Sym}

\DeclareMathOperator{\ab}{ab}

\DeclareMathOperator{\Sel}{Sel}
\DeclareMathOperator{\GSp}{GSp}

\fancypagestyle{plain}{
\fancyhf{} 
\fancyfoot[L]{\footnotesize{Department of Mathematics and Statistics, Boston University
		This research was partially supported by a doctoral research scholarship from the FQRNT}}

}
\DeclareTextCommand{\textthreeinferior}{PU}{\9040\203}
\DeclareTextCommand{\textrho}{PU}{\83\301}

\newcommand{\comdone}[1]{}

\numberwithin{equation}{section}

\renewcommand{\phi}{\varphi}

\newcommand{\bSel}{\overline{\operatorname{Sel}}}
\newcommand{\bL}{\overline{L}}

\newcommand{\chic}{\chi_p}	

\newcommand{\hone}[3][]{H^1_{#1}\!\left(#2,#3\right)}

\renewcommand{\mod}{\operatorname{mod}}
\renewcommand{\det}{\mathrm{det}}

\DeclareMathOperator{\exc}{exc}
\DeclareMathOperator{\loc}{loc}
\DeclareMathOperator{\glob}{glob}
\DeclareMathOperator{\Std}{Std}

\newcommand{\Hgx}{H_{\glob}^{\exc}}
\newcommand{\Hlx}{H_{\loc}^{\exc}}

\hypersetup{pdfauthor={Robert Harron},pdftitle={On Greenberg's \textit{L}-invariant of the symmetric sixth power of an ordinary cusp form},pdfkeywords={\textit{L}-invariants, \textit{p}-adic \textit{L}-functions, Hida families}}

\begin{document}
\author{Robert Harron}
\title{On Greenberg's $L$-invariant of the symmetric sixth power of an ordinary cusp form}
\date{\today}
\maketitle
\thispagestyle{plain}

\begin{abstract}
	We derive a formula for Greenberg's $L$-invariant of Tate twists of the symmetric sixth power of an ordinary non-CM cuspidal newform of weight $\geq4$, under some technical assumptions. This requires a ``sufficiently rich'' Galois deformation of the symmetric cube which we obtain from the symmetric cube lift to $\GSp(4)_{/\QQ}$ of Ramakrishnan--Shahidi and the Hida theory of this group developed by Tilouine--Urban. The $L$-invariant is expressed in terms of derivatives of Frobenius eigenvalues varying in the Hida family. Our result suggests that one could compute Greenberg's $L$-invariant of all symmetric powers by using appropriate functorial transfers and Hida theory on higher rank groups.
\end{abstract}

\tableofcontents
\newpage
\section*{Introduction}\addcontentsline{toc}{section}{Introduction}
	The notion of an $L$-invariant was introduced by Mazur, Tate, and Teitelbaum in their investigations of a $p$-adic analogue of the Birch and Swinnerton-Dyer conjecture in \cite{MTT86}. When considering the $p$-adic $L$-function of an elliptic curve $E$ over $\QQ$ with split, multiplicative reduction at $p$, they saw that its $p$-adic $L$-function vanishes even when its usual $L$-function does not (an ``exceptional zero'' or ``trivial zero''). They introduced a $p$-adic invariant, the ``($p$-adic) $L$-invariant'', of $E$ as a fudge factor to recuperate the $p$-adic interpolation property of $L(1,E,\chi)$ using the \textit{derivative} of its $p$-adic $L$-function. Their conjecture appears in \cite[\S\S13--14]{MTT86} and was proved by Greenberg and Stevens in \cite{GS93}. The proof conceptually splits up into two parts. One part relates the $L$-invariant of $E$ to the derivative in the ``weight direction'' of the unit eigenvalue of Frobenius in the Hida family containing $f$ (the modular form corresponding to $E$). The other part uses the functional equation of the two-variable $p$-adic $L$-function to relate the derivative in the weight direction to the derivative of interest, in the ``cyclotomic direction''. In this paper, we provide an analogue of the first part of this proof replacing the $p$-adic Galois representation $\rho_f$ attached to $f$ with Tate twists of $\Sym^6\!\rho_f$. More specifically, we obtain a formula for Greenberg's $L$-invariant (\cite{G94}) of Tate twists of $\Sym^6\!\rho_f$ in terms of derivatives in weight directions of the unit eigenvalues of Frobenius varying in some ordinary Galois deformation of $\Sym^3\!\rho_f$.

	Let us describe the previous work in this subject. In his original article, Greenberg (\cite{G94}) computes his $L$-invariant for all symmetric powers of $\rho_f$ when $f$ is associated to an elliptic curve with split, multiplicative reduction at $p$. In this case, the computation is local, and quite simple. In a series of articles, Hida has relaxed the assumption on $f$ allowing higher weights and dealing with Hilbert modular forms (see \cite{Hi07}), but still requiring, for the most part, $\rho_f$ to be (potentially) non-cristalline (though semistable) at $p$ in order to obtain an explicit formula for the $L$-invariant. A notable exception where a formula is known in the cristalline case is the symmetric square, done by Hida in \cite{Hi04} (see also chapter 2 of the author's Ph.D.\ thesis \cite{H-PhD} for a slightly different approach). Another exception comes again from Greenberg's original article (\cite{G94}) where he computes his $L$-invariant when $E$ has good ordinary reduction at $p$ \textit{and has complex multiplication}. In this case, the symmetric powers are reducible and the value of the $L$-invariant comes down to the result of Ferrero--Greenberg (\cite{FeG78}). The general difficulty in the cristalline case is that Greenberg's $L$-invariant is then a \textit{global} invariant and its computation requires the construction of a global Galois cohomology class.

	In this article, we attack the cristalline case for the next symmetric power which has an $L$-invariant, namely the sixth power (a symmetric power $n$ has an $L$-invariant in the cristalline case only when $n\equiv2\ (\mod 4)$). In general, one could expect to be able to compute Greenberg's $L$-invariant of $\Sym^n\!\rho_f$ by looking at ordinary Galois deformations of $\Sym^{\frac{n}{2}}\!\rho_f$ (see \S\ref{sec:sympowgeneral}). Unfortunately, when $n>2$ in the cristalline case, the $\Sym^{\frac{n}{2}}$ of the Hida deformation of $\rho_f$ is insufficient. The new ingredient we bring to the table is the idea to use a functorial transfer of $\Sym^{\frac{n}{2}}\!f$ to a higher rank group, use Hida theory there, and hope that the additional variables in the Hida family provide non-trivial Galois cohomology classes. In theorem \ref{thm:theoremA}, we show that this works for $n=6$ using the symmetric cube lift of Ramakrishnan--Shahidi (\cite{RS07}) (under certain technical assumptions). This provides hope that such a strategy would yield formulas for Greenberg's $L$-invariant for all symmetric powers in the cristalline case. The author is currently investigating if the combined use of the potential automorphy results of \cite{BLGHT09}, the functorial descent to a unitary group, and Hida theory on it (\cite{Hi02}) will be of service in this endeavour.

	We also address whether the $L$-invariant of the symmetric sixth power equals that of the symmetric square. There is a guess, due to Greenberg, that it does. We fall short of providing a definitive answer, but obtain a relation between the two in theorem \ref{thm:theoremB}.

	There are several facets of the symmetric sixth power $L$-invariant which we do not address. We do not discuss the expected non-vanishing of the $L$-invariant nor its expected relation to the size of a Selmer group. Furthermore, we make no attempt to show that Greenberg's $L$-invariant is the \textit{actual} $L$-invariant appearing in an interpolation formula of $L$-values. Aside from the fact that the $p$-adic $L$-function of the symmetric sixth power has not been constructed, a major impediment to proving this identity is that the point at which the $p$-adic $L$-function has an exceptional zero is no longer the centre of the functional equation, and a direct generalization of the second part of the proof of Greenberg--Stevens is therefore not possible. Citro suggests a way for dealing with this latter problem in the symmetric square case in \cite{Ci08}. Finally, we always restrict to the case where $f$ is ordinary at $p$. Recently, in \cite{Be09}, Benois has generalized Greenberg's definition of $L$-invariant to the non-ordinary case, and our results suggest that one could hope to compute his $L$-invariant using the eigenvariety for $\GSp(4)_{/\QQ}$.

	We remark that the results of this article were obtained in the author's Ph.D.\ thesis (\cite[Chapter 3]{H-PhD}). There, we give a slightly different construction of the global Galois cohomology class, still using the same deformation of the symmetric cube. In particular, we use Ribet's method of constructing a global extension of Galois representations by studying an irreducible, but residually reducible, representation. We refer to \cite{H-PhD} for details.

\begin{acknowledgements}
	I would like to thank my Ph.D.~advisor Andrew Wiles for help along the way. Thanks are also due to Chris Skinner for several useful conversations.
\end{acknowledgements}

\section*{Notation and conventions}\addcontentsline{toc}{section}{Notation and conventions}
	We fix throughout a prime $p\geq3$ and an isomorphism $\iota_\infty:\ol{\QQ}_p\cong\CC$. For a field $F$, $G_F$ denotes the absolute Galois group of $F$. We fix embeddings $\iota_\ell$ of $\ol{\QQ}$ into $\ol{\QQ}_\ell$ for all primes $\ell$. These define primes $\ol{\ell}$ of $\ol{\QQ}$ over $\ell$, and we let $G_\ell$ denote the decomposition group of $\ol{\ell}$ in $G_\QQ$, which we may thus identify with $G_{\QQ_\ell}$. Let $I_\ell$ denote the inertia subgroup of $G_\ell$. Let $\AA$ denote the adeles of $\QQ$ and let $\AA_f$ be the finite adeles.

	By a $p$-adic representation (over $K$) of a topological group $G$, we mean a continuous representation $\rho:G\rightarrow\Aut_K(V)$, where $K$ is a finite extension of $\Qp$ and $V$ is a finite-dimensional $K$-vector space equipped with its $p$-adic topology. Let $\chic$ denote the $p$-adic cyclotomic character. We denote the Tate dual $\Hom(V,K(1))$ of $V$ by $V^\ast$. Denote the Galois cohomology of the absolute Galois group of $F$ with coefficients in $M$ by $H^i(F,M)$.

	For compatibility with \cite{G94}, we take $\Frob_p$ to be an \textit{arithmetic} Frobenius element at $p$, and we normalize the local reciprocity map $\rec:\Qp^\times\rightarrow G^{\ab}_{\Qp}$ so that $\Frob_p$ corresponds to $p$. We normalize the $p$-adic logarithm $\log_p:\overline{\QQ}_p^\times\longrightarrow\overline{\QQ}_p$ by $\log_p(p)=0$.

\section{Greenberg's theory of trivial zeroes}
	In \cite{G94}, Greenberg introduced a theory describing the expected order of the trivial zero, as well as a conjectural value for the $L$-invariant of a $p$-ordinary motive. In this section, we briefly describe this theory, restricting ourselves to the case we will require in the sequel, specifically, we will assume the ``exceptional subquotient'' $W$ is isomorphic to the trivial representation. We end this section by explaining our basic method of computing $L$-invariants of symmetric powers of cusp forms.

\subsection{Ordinarity, exceptionality, and some Selmer groups}
	Let $\rho:G_\QQ\rightarrow\gl(V)$ be a $p$-adic representation over a field $K$. Recall that $V$ is called \textit{ordinary} if there is a descending filtration $\{F^iV\}_{i\in\ZZ}$ of $G_p$-stable $K$-subspaces of $V$ such that $I_p$ acts on $\gr^i\!V=F^iV/F^{i+1}V$ via multiplication by $\chic^i$ (and $F^iV=V$ (resp.\ $F^iV=0$) for $i$ sufficiently negative (resp.\ sufficiently positive)). Under this assumption, Greenberg (\cite{G89}) has defined what we call the \textit{ordinary Selmer group} for $V$ as
	\[ \Sel_\QQ(V):=\ker\!\left(\hone{\QQ}{V}\longrightarrow\prod_v\hone{\QQ_v}{V}/L_v(V)\right)
	\]
	where the product is over all places $v$ of $\QQ$ and the local conditions $L_v(V)$ are given by
	\begin{equation}\label{eqn:ordlocalconditions}
		L_v(V):=\left\{\begin{array}{ll}
					\hone[\nr]{\QQ_v}{V}:=\ker\!\left(\hone{\QQ}{V}\rightarrow\hone{I_v}{V}\right), & v\neq p\\
					\hone[\ord]{\QQ_p}{V}:=\ker\!\left(\hone{\QQ}{V}\rightarrow\hone{I_p}{V/F^1V}\right), & v=p.
				\end{array}\right.
	\end{equation}
	This Selmer group is conjecturally related to the $p$-adic $L$-function of $V$ at $s=1$.

	To develop the theory of exceptional zeroes following Greenberg (\cite{G94}), we introduce three additional assumptions on $V$ (which will be satisfied by the $V$ in which we are interested). Assume
	\begin{itemize}
		\item[(C)]\label{item:C} $V$ is \textit{critical} in the sense that $\dim_KV/F^1V=\dim_KV^-$, where $V^-$ is the $(-1)$-eigenspace of complex conjugation,
		\item[(U)]\label{item:U} $V$ has no $G_p$ subquotient isomorphic to the cristalline extension of $K$ by $K(1)$,
		\item[(S)]\label{item:S} $G_p$ acts semisimply on $gr^iV$ for all $i\in\ZZ$.
	\end{itemize} 
	If $V$ arises from a motive, condition (\hyperref[item:C]{C}) is equivalent to that motive being critical at $s=1$ in the sense of Deligne \cite{D79} (see \cite[\S6]{G89}). Condition (\hyperref[item:U]{U}) will come up when we want to define the $L$-invariant. Assumption (\hyperref[item:S]{S}) allows us to refine the ordinary filtration and define a $G_p$-subquotient of $V$ that (conjecturally) regulates the behaviour of $V$ with respect to exceptional zeroes.
	\begin{definition}\mbox{}
		\begin{enumerate}
			\item Let $F^{00}V$ be the maximal $G_p$-subspace of $F^0V$ such that $G_p$ acts trivially on $F^{00}V/F^1V$.
			\item Let $F^{11}V$ be the minimal $G_p$-subspace of $F^1V$ such that $G_p$ acts on $F^1V/F^{11}V$ via multiplication by $\chic$.
			\item Define the \textit{exceptional subquotient} $W$ of $V$ as
				\[ W:=F^{00}V/F^{11}V.
				\]
			\item $V$ is called \textit{exceptional} if $W\neq0$.
		\end{enumerate}
	\end{definition}
	Note that $W$ is ordinary with $F^2W=0, F^1W=F^1/F^{11}V$, and $F^0W=W$. For $?=00,11,$ or $i\in\ZZ$, we denote
	\[ F^?\hone{\Qp}{V}:=\im\!\left(\hone{\Qp}{F^?V}\longrightarrow\hone{\Qp}{V}\right).
	\]

	For simplicity, we impose the following condition on $V$ which will be sufficient for our later work:
	\begin{itemize}
		\item[(T$^\prime$)]\label{item:Tp} $W\cong K$, i.e.\ $F^{11}V=F^1V$ and $\dim_KF^{00}V/F^1V=1$.
	\end{itemize}
	We remark that this is a special case of condition (T) of \cite{G94}.

	The ordinarity of $V$ and assumptions (\hyperref[item:C]{C}), (\hyperref[item:U]{U}), (\hyperref[item:S]{S}), and (\hyperref[item:Tp]{T$^\prime$}) allow us to introduce Greenberg's \textit{balanced Selmer group} $\bSel_\QQ(V)$ of $V$ (terminology due to Hida) as follows. The local conditions $\bL_v(V)$ at $v\neq p$ are simply given by the unramified conditions $L_v(V)$ of \eqref{eqn:ordlocalconditions}. At $p$, $\bL_p(V)$ is characterized by the following two properties:
	\begin{itemize}
		\item[(Bal1)]\label{item:Bal1} $F^{11}\hone{\Qp}{V}\subseteq\bL_p(V)\subseteq F^{00}\hone{\Qp}{V}$,
		\item[(Bal2)]\label{item:Bal2} $\im\!\left(\bL_p(V)\rightarrow\hone{\Qp}{W}\right)=\hone[\nr]{\Qp}{W}$.
	\end{itemize}
	The balanced Selmer group of $V$ is
	\[ \bSel_\QQ(V):=\ker\!\left(\hone{\QQ}{V}\longrightarrow\prod_v\hone{\QQ_v}{V}/\bL_v(V)\right).
	\]

	The rationale behind the name ``balanced'' is provided by the following basic result of Greenberg's.
	\begin{proposition}[Proposition 2 of \cite{G94}]\label{prop:balanced}
		The balanced Selmer groups of $V$ and $V^\ast$ have the same dimension.
	\end{proposition}

	To make the reader feel more familiar with the balanced Selmer group, we offer the following result on its value under our running assumptions.
	\begin{proposition}
		Let $V$ be an ordinary $p$-adic representation of $G_\QQ$. Under assumptions \textup{(\hyperref[item:C]{C})}, \textup{(\hyperref[item:U]{U})}, \textup{(\hyperref[item:S]{S})}, and especially \textup{(\hyperref[item:Tp]{T$^\prime$})}, we have the following equalities
		\[ \bSel_\QQ(V)=\Sel_\QQ(V)=\hone[g]{\QQ}{V}=\hone[f]{\QQ}{V}
		\]
		where $\hone[g]{\QQ}{V}$ and $\hone[f]{\QQ}{V}$ are the Bloch--Kato Selmer groups introduced in \textup{\cite{BK90}}.
	\end{proposition}
	\begin{proof}
		The second equality is due to Flach (\cite[Lemma 2]{Fl90}) and the last equality follows from \cite[Corollary 3.8.4]{BK90}. We proceed to prove the first equality. The local conditions at $v\neq p$ are the same for $\bSel_\QQ(V)$ and $\Sel_\QQ(V)$ so we are left to show that $\bL_p(V)=L_p(V)$.

		Let $c\in\bL_p(V)$. Condition (\hyperref[item:Bal1]{Bal1}) implies that there is $c^\prime\in\hone{\Qp}{F^{00}V}$ mapping to $c$. By (\hyperref[item:Bal2]{Bal2}), the image of $c^\prime$ under the map in the bottom row of the commutative diagram
		\[
			\xymatrix{\hone{\Qp}{V}\ar@{->}[r] &\hone{I_p}{V/F^1V} \\
								\hone{\Qp}{F^{00}V}\ar@{->}[u] \ar@{->}[r] & \hone{I_p}{W}\ar@{->}[u]
			}
		\]
		is zero. Thus, $c$ is in the kernel of the map in the top row, which is exactly $L_p(V)$.

		For the reverse equality, let $c\in L_p(V)$ and consider the commutative diagram
		\[\xymatrix{\hone{\Qp}{V/F^{00}V}\ar@{->}[r]^{f_2} & \hone{I_p}{V/F^{00}V} \\
				c\in\hone{\Qp}{V}\ar@{->}[r]^{f_1} \ar@{->}[u]^{f_3} & \hone{I_p}{V/F^1V} \ar@{->}[u] \\
				\hone{\Qp}{F^{00}V}\ar@{->}[u] \ar@{->}[r] &\hone{I_p}{W}. \ar@{->}[u]^{f_4}
		}\]
		The local condition $L_p(V)$ satisfies (\hyperref[item:Bal1]{Bal1}) if $c\in\ker f_3$. By definition, $c\in\ker f_1$, so we show that $\ker f_2=0$. By inflation-restriction, $\ker f_2$ is equal to
		\[ \im\left(\hone{G_p/I_p}{\left(V/F^{00}V\right)^{I_p}}\longrightarrow\hone{\Qp}{V/F^{00}V}\right).
		\]
		Note that $\left(V/F^{00}V\right)^{I_p}=F^0V/F^{00}V$. The pro-cyclic group $G_p/I_p$ has (topological) generator $\Frob_p$, so
		\[ \hone{G_p/I_p}{F^0V/F^{00}V}\cong(F^0V/F^{00}V)/\left((\Frob_p-1)(F^0V/F^{00}V)\right)=0
		\]
		where the last equality is because $F^{00}V$ was defined to be exactly the part of $F^0V$ on which $\Frob_p$ acts trivially (mod $F^1V$). Thus, $L_p(V)$ satisfies (\hyperref[item:Bal1]{Bal1}), so there is a $c^\prime\in\hone{\Qp}{F^{00}V}$ mapping to $c$. Its image in $\hone{I_p}{V/F^1V}$ is trivial, so it suffices to show that $\ker f_4=0$ to conclude that $L_p(V)$ satisfies (\hyperref[item:Bal2]{Bal2}). By the long exact sequence in cohomology, the exactness (on the right) of
		\[ 0\longrightarrow W^{I_p}\longrightarrow(V/F^1V)^{I_p}\longrightarrow(V/F^{00}V)^{I_p}\longrightarrow0
		\]
		shows that $\ker f_4=0$.
	\end{proof}
	\begin{remark}
		In fact, this result is still valid if (\hyperref[item:Tp]{T$^\prime$}) is relaxed to simply $F^{11}V=F^1V$ (see \cite[Lemma 1.3.4]{H-PhD}).
	\end{remark}

\subsection{Greenberg's \texorpdfstring{$L$}{\textit{L}}-invariant}\label{sec:GrLinvar}
	We now proceed to define Greenberg's $L$-invariant. To do so, we impose one final condition on $V$, namely
	\begin{itemize}
		\item[(Z)]\label{item:Z} the balanced Selmer group of $V$ is zero: $\bSel_\QQ(V)=0$.
	\end{itemize}
	This will allow us to define a one-dimensional global subspace $\Hgx$ in a global Galois cohomology group (via some local conditions) whose image in $\hone{\Qp}{W}$ will be a line. The slope of this line is the $L$-invariant of $V$.

	Let $\Sigma$ denote the set of primes of $\QQ$ ramified for $V$, together with $p$ and $\infty$, let $\QQ_\Sigma$ denote the maximal extension of $\QQ$ unramified outside $\Sigma$, and let $G_\Sigma:=\Gal(\QQ_\Sigma/\QQ)$. By definition, $\bSel_\QQ(V)~\subseteq~\hone{G_\Sigma}{V}$. The Poitou--Tate exact sequence with local conditions $\bL_v(V)$ yields the exact sequence
	\[ 0\longrightarrow\bSel_\QQ(V)\longrightarrow\hone{G_\Sigma}{V}\longrightarrow\bigoplus_{v\in\Sigma}\hone{\QQ_v}{V}/\bL_v(V)\longrightarrow\bSel_\QQ(V^\ast).
	\]
	Combining this with assumption (\hyperref[item:Z]{Z}) and proposition \ref{prop:balanced} gives an isomorphism
	\begin{equation}\label{eqn:globalisom}
		\hone{G_\Sigma}{V}\cong\bigoplus_{v\in\Sigma}\hone{\QQ_v}{V}/\bL_v(V).
	\end{equation}
	\begin{definition}
		Let $\Hgx$ be the one-dimensional subspace\footnote{This is the subspace denoted $\wt{\mathbf{T}}$ in \cite{G94}. Page 161 of \textit{loc.\ cit.} shows that it is one-dimensional.} of $\hone{G_\Sigma}{V}$ corresponding to the subspace $F^{00}\hone{\Qp}{V}/\bL_p(V)$ of $\bigoplus_{v\in\Sigma}\hone{\QQ_v}{V}/\bL_v(V)$ under the isomorphism in (\ref{eqn:globalisom}).
	\end{definition}

	By definition of $F^{00}V$, we know that $(V/F^{00}V)^{G_p}=0$. Hence, we have injections
	\[ \hone{\Qp}{F^{00}V}\hookrightarrow\hone{\Qp}{V}
	\]
	and
	\[ \hone{\Qp}{W}\hookrightarrow\hone{\Qp}{V/F^1V}.
	\]
	\begin{definition}
		Let $\Hlx\subseteq\hone{\Qp}{W}$ be the image of $\Hgx$ in the bottom right cohomology group in the commutative diagram
		\[
			\xymatrix{\hone{G_\Sigma}{V}\ar@{->}[r]	&\hone{\Qp}{V}\ar@{->}[r]	&\hone{\Qp}{V/F^1V} \\
					\Hgx\ar@{->}[r]\ar@{}[u]|<<<{\LARGE \begin{rotate}{90}$\subseteq$\end{rotate}}	&F^{00}\hone{\Qp}{V}\ar@{}[u]|<<<{\LARGE \begin{rotate}{90}$\subseteq$\end{rotate}} \\
					&	\hone{\Qp}{F^{00}V}\ar@{->}[r]\ar@{}[u]|<<<{\LARGE \begin{rotate}{90}$\cong$\end{rotate}}	&\hone{\Qp}{W}.\ar@{^(->}[uu]
			}
		\]
	\end{definition}
	\begin{lemma}\mbox{}
		\begin{enumerate}
			\item $\dim_K\Hlx=1$,
			\item $\Hlx\cap\hone[\nr]{\Qp}{W}=0$.
		\end{enumerate}
	\end{lemma}
	\begin{proof}
		This follows immediately from the definitions of $\Hgx$ and of $\bL_p(V)$, together with assumption (\hyperref[item:U]{U}).
	\end{proof}

	There are canonical coordinates on $\hone{\Qp}{W}\cong\Hom(G_{\Qp},W)$ given as follows. Every homomorphism $\phi:G_{\Qp}\rightarrow W$ factors through the maximal pro-$p$ quotient of $G^{\ab}_{\Qp}$, which is $\Gal(\mathbf{F}_\infty/\Qp)$, where $\mathbf{F}_\infty$ is the compositum of two $\Zp$-extensions of $\Qp$: the cyclotomic one, $\QQ_{p,\infty}$, and the maximal unramified abelian extension $\Qp^{\nr}$. Let
	\[ \Gamma_\infty:=\Gal(\QQ_{p,\infty},\Qp)\cong\Gal(\mathbf{F}_\infty,\Qp^{\nr})
	\]
	and
	\[ \Gamma_{\nr}:=\Gal(\Qp^{\nr},\Qp)\cong\Gal(\mathbf{F}_\infty,\QQ_{p,\infty}),
	\]
	then
	\[ \Gal(\mathbf{F}_\infty,\Qp)=\Gamma_\infty\times\Gamma_{\nr}.
	\]
	Therefore, $\hone{\Qp}{W}$ breaks up into $\Hom(\Gamma_\infty,W)\times\Hom(\Gamma_{\nr},W)$. We have
	\[\Hom(\Gamma_\infty,W)=\Hom(\Gamma_\infty,\Qp)\otimes W		
	\]
	and
	\[\Hom(\Gamma_{\nr},W)=\Hom(\Gamma_{\nr},\Qp)\otimes W.		
	\]
	Composing the $p$-adic logarithm with the cyclotomic character provides a natural basis of\linebreak	
	$\Hom(\Gamma_\infty,\Qp)$, and the function $\ord_p:\Frob_p\mapsto1$ provides a natural basis of $\Hom(\Gamma_{\nr},\Qp)$. Coordinates are then provided by the isomorphisms
	\[ \func{\Hom(\Gamma_\infty,W)}{W}{\log_p\chic\otimes w}{w}
	\]
	and
	\[ \func{\Hom(\Gamma_{\nr},W)}{W}{\ord_p\otimes w}{w.}
	\]
	The $L$-invariant of $V$ is the slope of $\Hlx$ with respect to these coordinates.

	Specifically, we will compute the $L$-invariant in section \ref{sec:calcLinvar} by constructing a global class\linebreak	
	$[c]~\in~\hone{\QQ}{V}$ satisfying
	\begin{itemize}
		\item[(CL1)]\label{item:CL1} $[c_v]\in\hone[\nr]{\QQ_v}{V}$ for all $v\in\Sigma\setminus\{p\}$,
		\item[(CL2)]\label{item:CL2} $[c_p]\in F^{00}\hone{\Qp}{V}$,
		\item[(CL3)]\label{item:CL3} $[c_p]\not\equiv0\ \mathrm{mod}\ F^1V$.
	\end{itemize}
	This class then generates $\Hgx$, so its image $[\ol{c}_p]\in\hone{\Qp}{W}$ generates $\Hlx$. Let $u\in\Zp^\times$ be any principal unit, so that under our normalizations, $\chic(\rec(u))=u^{-1}$. Then, the coordinates of $[\ol{c}_p]$ are given by
	\begin{equation}\label{eqn:coords}
		\left(-\frac{1}{\log_pu}\ol{c}_p(\rec(u)),\ol{c}_p(\Frob_p)\right)
	\end{equation}
	where $\ol{c}_p$ is a cocyle in $[\ol{c}_p]$. Note that these coordinates are independent of the choice of $u$. We then have the following formula for the $L$-invariant of $V$:
	\begin{equation}\label{eqn:linvarslop}
		\mathcal{L}(V)=\frac{\ol{c}_p(\Frob_p)}{-\frac{1}{\log_pu}\ol{c}_p(\rec(u))}.
	\end{equation}

\subsection{Symmetric power \texorpdfstring{$L$}{\textit{L}}-invariants of ordinary cusp forms}\label{sec:sympowgeneral}
	Let $f$ be a $p$-ordinary,\footnote{More specifically, $\iota_\infty$-ordinary, in the sense that $\ord_p(\iota_\infty^{-1}(a_p))=0$, where $a_p$ is the $p$th Fourier coefficient of $f$.} holomorphic, cuspidal, normalized newform of weight $k\geq2$, level $\Gamma_1(N)$ (prime to $p$), and trivial character. Let $E=\QQ(f)$ be the field generated by the Fourier coefficients of $f$. Let $\mf{p}_0|p$ be the prime of $E$ above $p$ corresponding to the fixed embedding $\iota_p$, and let\linebreak		
	$\rho_f:G_\QQ\rightarrow\gl(V_f)$ be the contragredient of the $\mf{p}_0$-adic Galois representation (occurring in \'etale cohomology) attached to $f$ by Deligne \cite{D71} on the two-dimensional vector space $V_f$ over $K:=E_{\mf{p}_0}$. Let $\alpha_p$ denote the root of $x^2-a_px+p^{k-1}$ which is a $p$-adic unit. The $p$-ordinarity assumption implies that
	\[ \rho_f|_{G_p}\sim\mtrx{\chic^{k-1}\delta^{-1}}{\phi}{0}{\delta}
	\]
	where $\delta$ is the unramified character sending $\Frob_p$ to $\alpha_p$ (\cite[Theorem 2.1.4]{W88}). Thus, $\rho_f$ is ordinary. Note that assumption (\hyperref[item:S]{S}) is automatically satisfied by all (Tate twists of) symmetric powers of $\rho_f$ since all graded pieces of the ordinary filtration are one-dimensional. For condition (\hyperref[item:U]{U}) to be violated, we have must have $k=2$ and $\alpha_p=1$, but the Hasse bound shows that this is impossible.
	\begin{lemma}\label{lem:exceptwists}
		If $(\Sym^n\!\rho_f)(r)$ is an exceptional, critical Tate twist of $\rho_f$, then $n\equiv2\ (\mod 4)$, $r=\frac{n}{2}(1-k)$ or $\frac{n}{2}(1-k)+1$, and $k$ is even. Furthermore, the exceptional subquotient is isomorphic to K or K(1), respectively.
	\end{lemma}
	\begin{proof}
		The critical Tate twists are listed in \cite[Lemma 3.3]{RgS08}. Determining those that are exceptional is a quick computation, noting that $\delta$ is non-trivial.\comdone{(???is it?)}
	\end{proof}
	For the Tate twist by $\frac{n}{2}(1-k)+1$, the exceptional subquotient is isomorphic to $K(1)$, a case we did not treat in the previous section. However, Greenberg defines the $L$-invariant of such a representation as the $L$-invariant of its Tate dual, whose exceptional subquotient is isomorphic to the trivial representation. In fact, the Tate dual of the twist by $\frac{n}{2}(1-k)+1$ is the twist by $\frac{n}{2}(1-k)$.
	Accordingly, let $m$ be a positive odd integer, $n:=2m$, $\rho_n:=(\Sym^n\!\rho_f)(m(1-k))$, and assume $k$ is even. We present a basic setup for computing Greenberg's $L$-invariant $\mathcal{L}(\rho_n)$ using a deformation of $\rho_m:=\Sym^m\!\rho_f$. The main obstacle in carrying out this computation is to find a ``sufficiently rich'' deformation of $\rho_m$ to obtain a non-trivial answer. We do so below in the case $n=6$ for non-CM $f$ (of weight $\geq4$) by transferring $\rho_3$ to $\GSp(4)_{/\QQ}$ and using a Hida deformation on this group. The case $n=2$ has been dealt with by Hida in \cite{Hi04} (see also \cite[Chapter 2]{H-PhD}).

	We need a lemma from the finite-dimensional representation theory of $\gl(2)$ whose proof we leave to the reader.
	\begin{lemma}\label{lem:enddecomp}
		Let $\Std$ denote the standard representation of $\gl(2)$. Then, for $m$ an odd positive integer, there is a decomposition
		\[ \End\left(\Sym^m\!\Std\right)\cong\bigoplus_{i=0}^m\left(\Sym^{2i}\!\Std\right)\otimes\det^{-i}.
		\]
	\end{lemma}
	Since $\det\rho_f=\chic^{k-1}$, this lemma implies that $\rho_n$ occurs as a (global) direct summand in $\End\rho_m$. A deformation of $\rho_m$ provides a class in $\hone{\QQ}{\End\rho_m}$. If its projection to $\hone{\QQ}{\rho_n}$ is non-trivial (and satisfies conditions (\hyperref[item:CL1]{CL1--3}) of the previous section), then it generates $\Hgx$ and can be used to compute $\mathcal{L}(\rho_n)$.

	An obvious choice of deformation of $\rho_m$ is the symmetric $m$th power of the Hida deformation of $\rho_f$. The cohomology class of this deformation has a non-trivial projection to $\hone{\QQ}{\rho_n}$ only when $m=1$ (i.e.\ $n=2$, the symmetric square). For larger $m$, a ``richer'' deformation is required. The aims of the remaining sections of this paper are to obtain such a deformation in the case $m=3$ ($n=6$) and to use it to find a formula for the $L$-invariant of $\rho_6$ in terms of derivatives of Frobenius eigenvalues varying in the deformation.

\section{Input from \texorpdfstring{$\GSp(4)_{/\QQ}$}{GSp(4)/\textbf{Q}}}
	We use this section to set up our notations and conventions concerning the group $\GSp(4)_{/\QQ}$, its automorphic representations, its Hida theory, and the Ramakrishnan--Shahidi symmetric cube lift from $\gl(2)_{/\QQ}$ to it. We only provide what is required for our calculation of the $L$-invariant of $\rho_6$.

\subsection{Notation and conventions}
	Let $V$ be a four-dimensional vector space over $\QQ$ with basis $\{e_1,\dots,e_4\}$ equipped with the symplectic form given by
	\[ J=\left(\begin{array}{cccc}
			&&&1\\
			&&1\\
			&-1\\
			-1
		\end{array}\right).
	\]
	Let $\GSp(4)$ be the group of symplectic similitudes of $(V,J)$, i.e.\ $g\in\gl(4)$ such that ${}^tgJg=\nu(g)J$ for some $\nu(g)\in\GG_m$. The stabilizer of the isotropic flag $0\subseteq\langle e_1\rangle\subseteq\langle e_1,e_2\rangle$ is the Borel subgroup $B$ of $\GSp(4)$ whose elements are of the form
	\newlength{\arrrow}
	\settowidth{\arrrow}{$\frac{c}{a}$}
	\renewcommand{\arraystretch}{1.2}
	\[
		\left(\begin{array}{x{\arrrow}x{\arrrow}x{\arrrow}x{\arrrow}}
			a&\ast&\ast&\ast\\
			&b&\ast&\ast\\
			&&\frac{c}{b}&\ast\\
			&&&\frac{c}{a}
		\end{array}\right).
	\]
	Writing an element of the maximal torus $T$ as
	\settowidth{\arrrow}{$\frac{\nu(t)}{t_2}$}
	\renewcommand{\arraystretch}{1.6}
	\[t=\left(\begin{array}{x{\arrrow}x{\arrrow}x{\arrrow}x{\arrrow}}
			t_1\\
			&t_2\\
			&&\frac{\nu(t)}{t_2}\\
			&&&\frac{\nu(t)}{t_1}
		\end{array}\right),
	\]
	we identify the character group $X^\ast(T)$ with triples $(a,b,c)$ satisfying $a+b\equiv c\ (\mod\ 2)$ so that
	\[ t^{(a,b,c)}=t_1^{a}t_2^{b}\nu(t)^{(c-a-b)/2}.
	\]
	The dominant weights with respect to $B$ are those with $a\geq b\geq0$. If $\Pi$ is an automorphic representation of $\GSp(4,\AA)$ whose infinite component $\Pi_\infty$ is a holomorphic discrete series, we will say $\Pi$ has weight $(a,b)$ if $\Pi_\infty$ has the same infinitesimal character as the algebraic representation of $\GSp(4)$ whose highest weight is $(a,b,c)$ (for some $c$). For example, a classical Siegel modular form of (classical) weight $(k_1,k_2)$ gives rise to an automorphic representation of weight $(k_1-3,k_2-3)$ under our normalizations.

\subsection{The Ramarkrishnan--Shahidi symmetric cube lift}
	We wish to move the symmetric cube of a cusp form $f$ to a cuspidal automorphic representation of $\GSp(4,\AA)$ in order use the Hida theory on this group to obtain an interesting Galois deformation of the symmetric cube of $\rho_f$. The following functorial transfer due to Ramakrishnan and Shahidi (\cite[Theorem A$^\prime$]{RS07}) allows us to do so in certain circumstances.
	\begin{theorem}[Ramakrishnan--Shahidi \cite{RS07}]
		Let $\pi$ be the cuspidal automorphic representation of $\gl(2,\AA)$ defined by a holomorphic, non-CM newform $f$ of even weight $k\geq2$, level $N$, and trivial character. Then, there is an irreducible cuspidal automorphic representation $\Pi$ of $\GSp(4,\AA)$ with the following properties
		\begin{enumerate}
			\item $\Pi_\infty$ is in the holomorphic discrete series with its $L$-parameter being the symmetric cube of that of $\pi$,
			\item\label{item:RSweight} $\Pi$ has weight $(2(k-2),k-2)$, trivial central character, and is unramified outside of $N$,
			\item $\Pi^K\neq0$ for some compact open subgroup $K$ of $\GSp(4,\AA_f)$ of level equal to the conductor of $\Sym^3\!\rho_f$,
			\item $L(s,\Pi)=L(s,\pi,\Sym^3)$, where $L(s,\Pi)$ is the dergree 4 spin $L$-function,
			\item\label{item:RSgeneric} $\Pi$ is weakly equivalent\footnote{Recall that ``weakly equivalent'' means that the local components are isomorphic for almost all places.} to a globally generic cuspidal automorphic representation,
			\item $\Pi$ is not CAP, nor endoscopic.\footnote{Recall that an irreducible, cuspidal, automorphic representation of $\GSp(4,\AA)$ is ``CAP'' if it is weakly equivalent to the induction of an automorphic representation on a proper Levi subgroup, and it is ``endoscopic'' if the local $L$-factors of its spin $L$-function are equal, at almost all places, to the product of the local $L$-factors of two cuspidal automorphic representations of $\gl(2,\AA)$ with equal central characters.}
		\end{enumerate}
	\end{theorem}
	We remark that the weight in part \ref{item:RSweight} can be read off from the $L$-parameter of $\Pi_\infty$ given in \cite[(1.7)]{RS07}. As for part \ref{item:RSgeneric}, note that the construction of $\Pi$ begins by constructing a globally generic representation on the bottom of page 323 of \cite{RS07}, and ends by switching, in the middle of page 326, the infinite component from the generic discrete series element of the archimedean $L$-packet to the holomorphic one. Alternatively, in \cite{Wei08}, Weissauer has shown that any non-CAP non-endoscopic irreducible cuspidal automorphic representation of $\GSp(4,\AA)$ is weakly equivalent to a globally generic cuspidal automorphic representation.

\subsection{Hida deformation of \texorpdfstring{$\rho_3$}{\textrho\textthreeinferior} on \texorpdfstring{$\GSp(4)_{/\QQ}$}{GSp(4)/\textbf{Q}}}\label{sec:rho3t}
	Let $f$ a $p$-ordinary, holomorphic, non-CM, cuspidal, normalized newform of even weight $k\geq4$, level $\Gamma_1(N)$ (prime to $p$), and trivial character. We have added the non-CM hypothesis to be able to use the Ramakrishnan--Shahidi lift.\footnote{This is not really an issue as the CM case is much simpler.} According to lemma \ref{lem:exceptwists}, we only need to consider even weights. The restriction $k\neq2$ is forced by problems with the Hida theory on $\GSp(4)_{/\QQ}$ in the weight $(0,0)$.

	Tilouine and Urban (\cite{TU99}, \cite{U01}, \cite{U05}), as well as Pilloni (\cite{Pi10}, building on Hida (\cite{Hi02})) have worked on developing Hida theory on $\GSp(4)_{/\QQ}$. In this section, we describe the consequences their work has on the deformation theory of $\rho_3=\Sym^3\!\rho_f$ (where $\rho_f$ is as described in section \ref{sec:sympowgeneral}).

	We begin by imposing two new assumptions:
	\begin{itemize}
		\item[(\'Et)]\label{item:et} the universal ordinary $p$-adic Hecke algebra of $\GSp(4)_{/\QQ}$ is \'etale over the Iwasawa algebra at the height one prime corresponding to $\Pi$;
		\item[(RAI)]\label{item:RAI} the representation $\rho_3$ is residually absolutely irreducible.
	\end{itemize}

	Considering $\rho_3$ as the $p$-adic Galois representation attached to the Ramakrishnan--Shahidi lift $\Pi$ of $f$, we obtain a ring $\mathcal{A}$ of $p$-adic analytic functions in two variables $(s_1,s_2)$ on some neighbourhood of the point $(a,b)=(2(k-2),k-2)\in\Zp^2$, a free rank four module $\mathcal{M}$ over $\mathcal{A}$, and a deformation $\wt{\rho}_3:G_\QQ\rightarrow\Aut_\mathcal{A}(\mathcal{M})$ of $\rho_3$ such that $\wt{\rho}_3(a,b)=\rho_3$ and
		\begin{equation}\label{eqn:rho3t}
			 \wt{\rho}_3|_{G_p}\sim\left(\begin{array}{cccc}
								\theta_1\theta_2\mu_1 & \xi_{12} & \xi_{13} & \xi_{14} \\
								&\theta_2\mu_2 & \xi_{23} & \xi_{24} \\
								&&\theta_1\mu_2^{-1} & \xi_{34} \\
								&&&\mu_1^{-1}
							\end{array}\right)
		\end{equation}
		where the $\mu_i$ are unramified, and
		\begin{eqnarray}
			\mu_1(a,b)&=&\delta^{-3}\label{eqn:mu1} \\
			\mu_2(a,b)&=&\delta^{-1}\label{eqn:mu2} \\
			\theta_1(s_1,s_2)&=&\chic^{s_2+1}\label{eqn:theta1}\\
			\theta_1(a,b)&=&\chic^{k-1} \\
			\theta_2(s_1,s_2)&=&\chic^{s_1+2}\label{eqn:theta2}\\
			\theta_2(a,b)&=&\chic^{2(k-1)}.
		\end{eqnarray}
	\begin{remark}
		Assumption (\hyperref[item:RAI]{RAI}) allows us to take the integral version of \cite[Theorem 7.1]{TU99} (see the comment of \textit{loc.\ cit.} at the end of \S7) and assumption (\hyperref[item:et]{\'Et}) says that the coefficients are $p$-adic analytic. The shape of $\wt{\rho}_3|_{G_p}$ can be seen as follows. That four distinct Hodge--Tate weights show up can be seen by using \cite[Lemma 3.1]{U01} and the fact that both $\Pi$ and the representation obtained from $\Pi$ by switching the infinite component are automorphic. Applying Theorem 3.4 of \textit{loc.\ cit.} gives part of the general form of $\wt{\rho}_3|_{G_p}$ (taking into account that we work with the contragredient). The form the unramified characters on the diagonal take is due to $\wt{\rho}_3|_{G_p}$ taking values in the Borel subgroup $B$ (this follows from Corollary 3.2 and Proposition 3.4 of \textit{loc.\ cit.}). That the specializations of the $\mu_i$ and $\theta_i$ at $(a,b)$ are what they are is simply because $\wt{\rho}_3$ is a deformation of $\rho_3$.
	\end{remark}

	We may take advantage of assumption (\hyperref[item:et]{\'Et}) to determine a bit more information about the $\mu_i$. Indeed, let $\wt{\rho}_f$ denote the Hida deformation (on $\gl(2)_{/\QQ}$) of $\rho_f$, so that
	\[ \wt{\rho}_f|_{G_p}\sim\mtrx{\theta\mu^{-1}}{\xi}{0}{\mu}
	\]
	where $\theta,\mu,$ and $\xi$ are $p$-adic analytic functions on some neighbourhood of $s=k$, $\theta(s)=\chic^{s-1}$, and $\mu(s)$ is the unramified character sending $\Frob_p$ to $\alpha_p(s)$ (where $\alpha_p(s)$ is the $p$-adic analytic function giving the $p$th Fourier coefficients in the Hida family of $f$) (\cite[Theorem 2.2.2]{W88}).
	By \cite[Remark 9]{GhVa04}, we know that every arithmetic specialization of $\wt{\rho}_f$ is non-CM. We may thus apply the Ramakrishnan--Shahidi lift\comdone{(???trivial character?)} to the even weight specializations and conclude that $\Sym^3\!\wt{\rho}_f$ is an ordinary modular deformation of $\rho_3$. Assumption (\hyperref[item:et]{\'Et}) then implies that $\Sym^3\!\wt{\rho}_f$ is a specialization of $\wt{\rho}_3$. Since the weights of the symmetric cube lift of a weight $k^\prime$ cusp form are $(2(k^\prime-2),k^\prime-2)$, we can conclude that $\Sym^3\!\wt{\rho}_f$ is the ``sub-family'' of $\wt{\rho}_3$ where $s_1=2s_2$. Thus,
	\begin{eqnarray*}
		\mu_1(2s,s)&=&\mu^{-3}(s+2)\\
		\mu_2(2s,s)&=&\mu^{-1}(s+2).
	\end{eqnarray*}
	Applying the chain rule yields
	\begin{eqnarray}
		2\partial_1\mu_1(a,b)+\partial_2\mu_1(a,b)&=&-\frac{3\mu^\prime(k)}{\delta^4}\label{eqn:mu1relations}\\
		2\partial_1\mu_2(a,b)+\partial_2\mu_2(a,b)&=&-\frac{\mu^\prime(k)}{\delta^2}.\label{eqn:mu2relations}
	\end{eqnarray}

\section{Calculating the \texorpdfstring{$L$}{\textit{L}}-invariant}\label{sec:calcLinvar}
	For the remainder of this article, let $f$ a $p$-ordinary, holomorphic, non-CM, cuspidal, normalized newform of even weight $k\geq4$, level $\Gamma_1(N)$ (prime to $p$), and trivial character. Let $\rho_f$, $\rho_3$, and $\rho_6$ be as in section \ref{sec:sympowgeneral}, and let $W$ denote the exceptional subquotient of $\rho_6$. Furthermore, assume condition (\hyperref[item:Z]{Z}) that $\bSel_\QQ(\rho_6)=0$. We now put together the ingredients of the previous sections to compute Greenberg's $L$-invariant of $\rho_6$.

\subsection{Constructing the global Galois cohomology class}
	Recall that if $\rho_3^\prime$ is an infinitesimal deformation of $\rho_3$ (over $K[\epsilon]:=K[x]/(x^2)$), a corresponding cocycle $c_3^\prime:G_\QQ\rightarrow\End\rho_3$ is defined by the equation
	\[\rho_3^\prime(g)=\rho_3(g)(1+\epsilon c_3^\prime(g)).
	\]
	Let $\wt{\rho}_3$ be the deformation of $\rho_3$ constructed in section \ref{sec:rho3t}. Taking a first order expansion of the entries of $\wt{\rho}_3$ around $(a,b)=(2(k-2),k-2))$ in any given direction yields an infinitesimal deformation of $\rho_3$. We parametrize these as follows. A $p$-adic analytic function $F\in\mathcal{A}$ has a first-order expansion near $(s_1,s_2)=(a,b)$ given by
	\[ F(a+\epsilon_1,b+\epsilon_2)\approx F(a,b)+\epsilon_1\partial_1F(a,b)+\epsilon_2\partial_2F(a,b).
	\]
	We introduce a parameter $\Delta$ such that $\epsilon_1=(1-\Delta)\epsilon$ and $\epsilon_2=\Delta\epsilon$. Let $\wt{\rho}_{3,\Delta}$ denote the infinitesimal deformation of $\rho_3$ obtained by first specializing $\wt{\rho}_3$ along the direction given by $\Delta$, then taking the quotient by the ideal $((s_1,s_2)-(a,b))^2$. Let $c_{6,\Delta}$ denote the projection of its corresponding cocyle to $\rho_6$ (in the decomposition of lemma \ref{lem:enddecomp}).

\subsection{Properties of the global Galois cohomology class}
	To use $c_{6,\Delta}$ to compute the $L$-invariant of $\rho_6$, we must show that it satisfies conditions (CL1--3) of section \ref{sec:GrLinvar}. The proofs of lemma 1.2 and 1.3 of \cite{Hi07} apply to the cocycle $c_{6,\Delta}$ to show that it satisfies (\hyperref[item:CL1]{CL1}).\footnote{The deformation $\wt{\rho}_3$ clearly satisfies conditions (K$_3$1--4) of \cite[\S0]{Hi07} and our $c_{6,\Delta}$ is a special case of the cocycles Hida defines in the proof of lemma 1.2 of \textit{loc.\ cit.}}

	To verify conditions (\hyperref[item:CL2]{CL2}) and (\hyperref[item:CL3]{CL3}) (and to compute the $L$-invariant of $\rho_6$), we need to find an explicit formula for part of $c_{6,\Delta}|_{G_p}$. We know that
	\[ \rho_f|_{G_p}\sim\mtrx{\chic^{k-1}\delta^{-1}}{\phi}{0}{\delta}.
	\]
	Taking the symmetric cube (considered as a subspace of the third tensor power) yields
	\settowidth{\arrrow}{$\chic^{3(k-1)}\delta^{-1}$}
	\renewcommand{\arraystretch}{2.2}
	\[ \rho_3|_{G_p}\sim\left(
		\begin{array}{x{\arrrow}x{\arrrow}x{\arrrow}x{\arrrow}}
			\chic^{3(k-1)}\delta^{-3}	& \frac{3\chic^{2(k-1)}\phi}{\delta^2}	& \frac{3\chic^{k-1}\phi^2}{\delta}	& \phi^3 \\
								& \chic^{2(k-1)}\delta^{-1}			& 2\chic^{k-1}\phi				& \delta\phi^2 \\
								&							& \chic^{k-1}\delta				& \delta^2\phi \\
								&							&							& \delta^3
		\end{array}\right).
	\]
	Taking first-order expansions of the entries of $\wt{\rho}_3\rho_3^{-1}-I_4$, specializing along the direction given by $\Delta$, and projecting yields $c_{6,\Delta}$. However, since we are interested in an explicit formula for $c_{6,\Delta}|_{G_p}$, we need to determine a basis that gives the decomposition of lemma \ref{lem:enddecomp}. This can be done using the theory of raising and lowering operators. We obtain the following result.
	\begin{theorem}
		In such an aforementioned basis,
		\begin{equation}\label{eqn:c6formula}
			c_{6,\Delta}|_{G_p}\sim\left(\begin{array}{c}
							\ast\\
							(1-\Delta)\left(\frac{\partial_1\theta_2}{\chic^{2(k-1)}}-\frac{2\partial_1\theta_1}{\chic^{k-1}}-\delta^3\partial_1\mu_1+3\delta\partial_1\mu_2\right) \\
							\hspace{0.28in}+\Delta\left(\frac{\partial_2\theta_2}{\chic^{2(k-1)}}-\frac{2\partial_2\theta_1}{\chic^{k-1}}-\delta^3\partial_2\mu_1+3\delta\partial_2\mu_2\right)\\
							0
					\end{array}\right)
		\end{equation}
		where $\ast$ and $0$ are both $3\times1$, and all derivatives are evaluated at $(a,b)$.
	\end{theorem}
	Since the bottom three coordinates in \eqref{eqn:c6formula} are zero, the image of $c_{6,\Delta}|_{G_p}$ lands in $F^{00}\rho_6$, i.e.\ $c_{6,\Delta}$ satisfies (\hyperref[item:CL2]{CL2}). If we can show that the middle coordinate is non-zero, then $c_{6,\Delta}$ satisfies (\hyperref[item:CL3]{CL3}). In fact, we will show that $c_{6,\Delta}$ satisfies (\hyperref[item:CL3]{CL3}) if, and only if, $\Delta\neq1/3$ (in this latter case, we will show that $[c_{6,1/3}]=0$).

	Let $\ol{c}_{6,\Delta}$ denote the image of $c_{6,\Delta}$ in $\hone{\Qp}{W}$. Let
	\[ \alpha_p^{(i,j)}:=\partial_j\mu_i(a,b)(\Frob_p).
	\]
	\begin{corollary}\label{prop:c6coords}
		The coordinates of $\ol{c}_{6,\Delta}$, as in equation \eqref{eqn:coords}, are
		\[ \left(1-3\Delta,(1-\Delta)\left(-\alpha_p^3\alpha_p^{(1,1)}+3\alpha_p\alpha_p^{(2,1)}\right)+\Delta\left(-\alpha_p^3\alpha_p^{(1,2)}+3\alpha_p\alpha_p^{(2,2)}\right)\right).
		\]
		In particular, if $\Delta\neq1/3$, then $c_{6,\Delta}$ satisfies (\hyperref[item:CL3]{CL3}).
	\end{corollary}
	Before proving this, we state and prove a lemma.
	\begin{lemma}
		Recall that $\theta(s)=\chic^{s-1}$. For any integer $s\geq2$, and any principal unit $u$,
		\begin{enumerate}
			\item $\theta^\prime(s)(\Frob_p)=0$,
			\item $\displaystyle{\frac{\theta^\prime(s)(\rec(u))}{\chic^{s-1}(\rec(u))}=-\log_pu}$.
		\end{enumerate}
	\end{lemma}
	\begin{proof}
		The first equality is simply because $\chic(\Frob_p)=1$. For the second, recall that\linebreak	
		$\chic(\rec(u))=u^{-1}$, so $\theta(s)(\rec(u))=u^{1-s}$. Thus, the logarithmic derivative of $\theta(s)(\rec(u))$ is indeed $-\log_pu$.
	\end{proof}
	\begin{proof}[Proof of corollary \ref{prop:c6coords}]
		The first coordinate is obtained by taking an arbitrary principal unit $u$, evaluating $\ol{c}_{6,\Delta}$ and $\rec(u)$ and dividing by $-\log_pu$. By equations \eqref{eqn:theta1} and \eqref{eqn:theta2}, $\partial_i\theta_i=0$. Combining the fact that the $\mu_i$ are unramified with part (b) of the above lemma yields
		\[ \frac{\ol{c}_{6,\Delta}(\rec(u))}{-\log_pu}=\frac{(1-\Delta)(-\log_pu)+\Delta(-2\log_pu)}{-\log_pu}=1-3\Delta.
		\]
		If $\Delta\neq1/3$, the first coordinate is non-zero, so $\ol{c}_{6,\Delta}$ itself is non-zero, so $c_{6,\Delta}$ satisfies (\hyperref[item:CL3]{CL3}).

		Combining part (a) of the above lemma with equations \eqref{eqn:mu1} and \eqref{eqn:mu2} yields the second coordinate (recall that $\delta(\Frob_p)=\alpha_p$).
	\end{proof}
	\begin{remark}
		If we take $\Delta=1/3$, the first coordinate of $\ol{c}_{6,1/3}$ vanishes. Hence,\linebreak	
		$\ol{c}_{6,1/3}~\in~\hone[\nr]{\Qp}{W}$. Therefore, $[c_{6,1/3}]\in\bSel_\QQ(V)=0$ (by assumption (\hyperref[item:Z]{Z})). The direction $\Delta=1/3$ is the one for which $\epsilon_1/\epsilon_2=2$, i.e.\ the direction corresponding to the symmetric cube of the $\gl(2)$ Hida deformation of $\rho_f$. This is an instance of the behaviour mentioned at the end of section \ref{sec:sympowgeneral}.
	\end{remark}

\subsection{Formula for the \texorpdfstring{$L$}{\textit{L}}-invariant}
	Tying all this together yields the main theorem of this article.
	\begin{theoremA}\label{thm:theoremA}
		Let $p\geq3$ be a prime. Let $f$ a $p$-ordinary, holomorphic, non-CM, cuspidal, normalized newform of even weight $k\geq4$, level $\Gamma_1(N)$ (prime to $p$), and trivial character. Let $\rho$ be a critical, exceptional Tate twist of $\Sym^6\!\rho_f$, i.e.\ $\rho=\rho_6=(\Sym^6\!\rho_f)(3(1-k))$ or its Tate dual. Assume conditions \textup{(\hyperref[item:et]{\'Et})}, \textup{(\hyperref[item:RAI]{RAI})}, and \textup{(\hyperref[item:Z]{Z})}. Then,
		\begin{equation}\label{eqn:sym6Linvar}
			\mathcal{L}(\rho)=-\alpha_p^3\alpha_p^{(1,1)}+3\alpha_p\alpha_p^{(2,1)}.
		\end{equation}
	\end{theoremA}
	\begin{proof}
		Pick any $\Delta\neq1/3$. We've shown that $[c_{6,\Delta}]$ satisfies (\hyperref[item:CL1]{CL1--3}) and hence generates $\Hgx$. The coordinates of its image in $\hone{\Qp}{W}$ were obtained in corollary \ref{prop:c6coords}. Therefore, $\mathcal{L}(\rho_6)$ can be computed from equation \eqref{eqn:linvarslop}. Specifically, the result is obtained by solving the system of linear equations in $\mathcal{L}(\rho_6)$ and the $\alpha_p^{(i,j)}$ given by the coordinates of $\ol{c}_{6,\Delta}$ and equations \eqref{eqn:mu1relations} and \eqref{eqn:mu2relations}. The $L$-invariant of $\rho_6^\ast$ is by definition that of $\rho_6$.
	\end{proof}
	\begin{remark}
		We could express this result in terms of other $\alpha_p^{(i,j)}$. For example, picking $\Delta=1$ yields
		\[ \mathcal{L}(\rho_6)=\frac{1}{2}\alpha_p^3\alpha_p^{(1,2)}-\frac{3}{2}\alpha_p\alpha_p^{(2,2)}.
		\]
	\end{remark}

\subsection{Relation to Greenberg's \texorpdfstring{$L$}{\textit{L}}-invariant of the symmetric square}
	We can carry out the above analysis for the projection to $\rho_2:=(\Sym^2\!\rho_f)(1-k)$ in lemma \ref{lem:enddecomp} and compare the value of $\mathcal{L}(\rho_2)$ obtained with the known value (\cite[Theorem 1.1]{Hi04}, \cite[Theorem~A]{H-PhD})
	\[ \mathcal{L}(\rho_2)=-2\frac{\alpha_p^\prime}{\alpha_p}
	\]
	where $\alpha_p^\prime=\mu^\prime(k)(\Frob_p)=\alpha_p^\prime(k)$, and one assumes that $\bSel_\QQ(\rho_2)=0$.\footnote{This vanishing is known in many cases due to work of Hida (\cite{Hi04}), Kisin (\cite{Ki04}), and Weston (\cite{Wes04}). See those papers for details or \cite[Theorem 2.1.1]{H-PhD} for a summary.} The restriction of the cocycle $c_{2,\Delta}$ (in an appropriate basis) is
	\[ c_{2,\Delta}|_{G_p}\sim\left(\begin{array}{c}
							\ast\\
							(1-\Delta)\left(-\frac{2\partial_1\theta_2}{\chic^{2(k-1)}}-\frac{\partial_1\theta_1}{\chic^{k-1}}-3\delta^3\partial_1\mu_1-\delta\partial_1\mu_2\right) \\
							\hspace{0.28in}+\Delta\left(-\frac{2\partial_2\theta_2}{\chic^{2(k-1)}}-\frac{\partial_2\theta_1}{\chic^{k-1}}-3\delta^3\partial_2\mu_1-\delta\partial_2\mu_2\right)\\
							0
					\end{array}\right)
	\]
	Accordingly, the coordinates of the class $\ol{c}_{2,\Delta}$ are
	\[ \left(\Delta-2,(1-\Delta)\left(-3\alpha_p^3\alpha_p^{(1,1)}-\alpha_p\alpha_p^{(2,1)}\right)+\Delta\left(-3\alpha_p^3\alpha_p^{(1,2)}-\alpha_p\alpha_p^{(2,2)}\right)\right).
	\]
	The cocycle $c_{2,\Delta}$ can be used to compute $\mathcal{L}(\rho_2)$ when $\Delta\neq2$. When $\Delta=2$, one has, as above, $[c_{2,\Delta}]\in\bSel_\QQ(\rho_2)$. Taking $\Delta=0$ yields
	\begin{equation}\label{eqn:sym2Linvar}
		\mathcal{L}(\rho_2)=\frac{3}{2}\alpha_p^3\alpha_p^{(1,1)}+\frac{1}{2}\alpha_p\alpha_p^{(2,1)}.
	\end{equation}
	Combining equations \eqref{eqn:sym6Linvar} and \eqref{eqn:sym2Linvar} yields the following relation between $L$-invariants.
	\begin{theoremA}\label{thm:theoremB}
		Assuming \textup{(\hyperref[item:et]{\'Et})}, \textup{(\hyperref[item:RAI]{RAI})}, \textup{(\hyperref[item:Z]{Z})}, and $\bSel_\QQ(\rho_2)=0$, we have
		\[ \mathcal{L}(\rho_6)=-10\alpha_p^3\alpha_p^{(1,1)}+6\mathcal{L}(\rho_2).
		\]
	\end{theoremA}
	\begin{remark}
		There is a guess, suggested by Greenberg \cite[p.~170]{G94}, that the $L$-invariants of all symmetric powers of $\rho_f$ should be equal. This is known in the cases where it is relatively easy to compute the $L$-invariant, namely when $f$ corresponds to an elliptic curve with split, multiplicative reduction at $p$, or when $f$ has CM. In the case at hand, we fall one relation short of showing the equality of $\mathcal{L}(\rho_6)$ and $\mathcal{L}(\rho_2)$. Equality would occur if one knew the relation
		\[ \alpha_p^{(1,1)}\overset{?}{=}-\frac{\alpha_p^\prime}{\alpha_p^4}.
		\]
	\end{remark}
\nocite{xxx}
\nocite{yyy}
\addcontentsline{toc}{section}{References}
\bibliographystyle{modamsalpha}
\bibliography{T_Master}

\end{document}